\documentclass{amsart}

\usepackage{amsmath}
\usepackage{amsfonts}
\usepackage{amssymb}
\usepackage{amsthm}
\usepackage{array}
\usepackage{bbm}
\usepackage{chngpage}
\usepackage{comment}
\usepackage{float}
\usepackage[OT2,T1]{fontenc}
\usepackage{graphicx,caption}
\usepackage[export]{adjustbox}
\usepackage{longtable}
\usepackage{mathrsfs}
\usepackage{mathtools}
\usepackage{wrapfig}
\usepackage{cutwin}
\usepackage{shapepar}
\usepackage{tikz}
\usepackage[lowtilde]{url}
\usepackage{subcaption}
\usetikzlibrary{matrix,arrows}
\usepackage{tabularx}
\usepackage{multirow}
\definecolor{citeclr}{rgb}{0.55, 0.55, 0.64}
\definecolor{linkclr}{rgb}{0, 0.21, 0.9447}
\usepackage[colorlinks,
  linkcolor=linkclr,
  citecolor=citeclr,
  urlcolor=blue!46!cyan,
  ]{hyperref}
\usepackage[nameinlink, nosort]{cleveref}
\usepackage{enumitem}
\usepackage{diagbox}
\usepackage{etoolbox}
\usepackage{algpseudocode}
\usepackage{etoolbox}

\usepackage[margin=0.9in]{geometry}

\patchcmd{\subsection}{-.5em}{.5em}{}{}

\patchcmd{\section}{\normalfont}{\normalfont\Large}{}{}

\newtheorem{theorem}{Theorem}[section]
\newtheorem{lemma}[theorem]{Lemma}

\newtheorem{statistic}[theorem]{Statistic}

\newtheorem{fact}[theorem]{Fact}

\theoremstyle{definition}

\newtheorem{remark}[theorem]{Remark}

\crefname{figurecap}{figure}{figures}
\newtheorem{tablecap}[theorem]{Table}
\crefname{tablecap}{table}{tables}

\crefname{idea}{idea}{ideas}

\crefname{observation}{observation}{observations}
\newtheorem{example}[theorem]{Example}
\AtBeginEnvironment{example}{%
  \pushQED{\qed}%
}
\AtEndEnvironment{example}{\popQED\endexample}

\newcommand{\E}{\mathbb{E}}

\newcommand{\PP}{\mathbb{P}}

\newcommand{\R}{\mathbb{R}}
\newcommand{\Z}{\mathbb{Z}}

\newcommand{\cT}{\mathcal{T}}

\newcommand{\fp}{\mathfrak{p}}
\newcommand{\fq}{\mathfrak{q}}

\DeclareSymbolFont{cyrletters}{OT2}{wncyr}{m}{n}
\DeclareMathSymbol{\sha}{\mathalpha}{cyrletters}{"58}

\newcommand{\eps}{\varepsilon}

\newcommand{\ds}{\displaystyle}
\newlength{\strutheight}
\newcommand{\half}{\frac{1}{2}}
\newcommand{\thalf}{\tfrac{1}{2}}

\renewcommand{\mod}{\mspace{4mu}\mathrm{mod}\mspace{4mu}}

\newcommand{\q}{q}
\newcommand{\X}{X}

\newcommand{\zqz}{(\Z/\q\Z)^\times}
\newcommand{\unif}{\mathrm{Unif}\zqz}
\newcommand{\dkl}{D_{\mathrm{KL}}}

\author{Alex Cowan}
\address{Department of Mathematics, University of Waterloo, Waterloo, ON, Canada}
\email{alex.cowan@uwaterloo.ca}

\title{The relative entropy of primes in arithmetic progressions is really small}
\date{\today}

\AtBeginDocument{%
   \def\MR#1{}
}

\begin{document}
\begin{abstract}
Fix a modulus $\q$. One would expect the number of primes in each invertible residue class mod $\q$  to be multinomially distributed, i.e.\ for each $p \mod q$ to behave like an independent random variable uniform on $\zqz$. Using techniques from data science, we discover overwhelming evidence to the contrary: primes are much more uniformly distributed than iid uniform random variables. This phenomenon was previously unknown, and there is no clear theoretical explanation for it.

To demonstrate that our test statistic of choice, the KL divergence, is indeed extreme, we prove new bounds for the left tail of the relative entropy of the uniform multinomial using the method of types.
\end{abstract}
\maketitle
\tableofcontents

\section{Models, statistics, and percentiles}

In a nebulous sense, it's a good rule of thumb that when prime numbers are ordered by their absolute values, their residues modulo some fixed $\q$ behave like independent uniform random variables on $\zqz$.
Let
\begin{align}
  \label{eq:xidef}
  \xi_{p,\q} \sim \unif
\end{align}
be iid random variables indexed by primes $p$.
The sequence
\begin{align}
  \label{eq:random_seq}
  \xi_{2,\q},\, \xi_{3,\q},\, \xi_{5,\q},\, \dots
\end{align}
is a \textit{model} for
\begin{align}
  \label{eq:arithmetic_seq}
  2\mod\q,\, 3\mod\q,\, 5\mod\q,\, \dots
\end{align}
How might we assess it?

The aforementioned rule of thumb can be recast as saying that the sequences \eqref{eq:arithmetic_seq} and \eqref{eq:random_seq}
produce similar statistics\footnote{The word ``statistics'' has two meanings: a discipline or body of methodologies, and the plural of ``statistic'': a quantity derived from data. Here our meaning is the latter.}.
We make sense of arithmetic statistics such as
\begin{align}
  \label{eq:8data}
  \begin{aligned}
  \#\{p < 10^8 \,:\, p = 1\mod 8\} &= 1,\mspace{-2.5mu}439,\mspace{-2.5mu}970
  \\
  \#\{p < 10^8 \,:\, p = 3\mod 8\} &= 1,\mspace{-2.5mu}440,\mspace{-2.5mu}544
  \\
  \#\{p < 10^8 \,:\, p = 5\mod 8\} &= 1,\mspace{-2.5mu}440,\mspace{-2.5mu}534
  \\
  \#\{p < 10^8 \,:\, p = 7\mod 8\} &= 1,\mspace{-2.5mu}440,\mspace{-2.5mu}406
  \end{aligned}
\end{align}
by determining to what degree they resemble the statistics typical realizations of the random model \eqref{eq:random_seq} would produce,
discovering for instance that the counts of \eqref{eq:8data} are more uniform than about $98.5\%$ of random samples.

To measure the extent to which arithmetic data behaves randomly we introduce the following framework.
Let $f$ be a function on sequences of the form \eqref{eq:arithmetic_seq} and \eqref{eq:random_seq}, with codomain $\R$ for simplicity.
When $f$ is evaluated on the arithmetic data \eqref{eq:arithmetic_seq}, the result is a real number $\tau$, an \textit{empirical statistic}. 
When $f$ is evaluated on \eqref{eq:random_seq}, the result is instead a random variable $T$, a \textit{random statistic}.
The extent to which arithmetic data such as \eqref{eq:arithmetic_seq} behaves randomly can be assessed by estimating the \textit{percentile}
\begin{align}
  \label{eq:percentile_def}
  \PP(T < \tau)
  =
  \PP\big( f(\xi_{2,\q}, \xi_{3,\q}, \xi_{5,\q},\dots) < f(2\mod\q, 3\mod\q, 5\mod\q, \dots) \big)
  ,
\end{align}
where $\PP$ is short for the probability of.
For the arithmetic data \eqref{eq:arithmetic_seq} to masquerade as random, the percentile \eqref{eq:percentile_def} shouldn't be too close to $0\%$ nor $100\%$. We now have a tool for discovering non-random behaviour.

\section{Example: Chebyshev's bias}
\label{sec:chebyshev}

A well-known instance in which
primes mod $\q$ do not behave randomly is ``Chebyshev's bias'' \cite{rubinstein_sarnak}.
Let's look at what is,
on the arithmetic side ``prime number races'',
and
on the statistical side ``stopping times of random walks''.
As usual, let
\begin{align*}
  \pi(x) &\coloneqq \#\{p \leqslant x\}
  \\
  \pi(x; a\mod\q) &\coloneqq \#\{p \leqslant x \,:\, p = a\mod\q\}
  .
\end{align*}

\begin{statistic}
  \label{statistic:chebyshev}
  On the arithmetic side, define
  \begin{align*}
    \tau_2 &\coloneqq \min \{x \,:\, \pi(x;\,1 \mod 4) > \pi(x;\,3 \mod 4)\}
    \\
    \tau_3 &\coloneqq \min \{x \,:\, \pi(x;\,1 \mod 3) > \pi(x;\,2 \mod 3)\}
    ,
  \end{align*}
  to be compared on the statistical side with
  \begin{align*}
    T_2 &\coloneqq \min \!\left\{x \,:\, \sum_{\substack{p \leqslant x \\ p \neq 2}} \begin{cases}\phantom{-}1 & \text{\emph{if $\xi_{p,4} = 1$}} \\ -1 & \text{\emph{if $\xi_{p,4} = 3$}}\end{cases} \,> 0\right\}
    \\
    T_3 &\coloneqq \min \!\left\{x \,:\, \sum_{\substack{p \leqslant x \\ p \neq 3}} \begin{cases}\phantom{-}1 & \text{\emph{if $\xi_{p,3} = 1$}} \\ -1 & \text{\emph{if $\xi_{p,3} = 2$}}\end{cases} \,> 0\right\}
  \end{align*}
  respectively.
\end{statistic}
Our task is to estimate $\PP(T_2 < \tau_2)$ and $\PP(T_3 < \tau_3)$.

Let $\X_k \sim 2\mathrm{Ber}(\thalf) - 1$, $k = 1, 2, \dots$ be iid random variables that are $\pm 1$ each with probability $1/2$, and let $S_t$ denote the simple random walk $\X_1 + \ldots + \X_t$. The percentile $\PP(T_2 < \tau_2)$ can be written 
\begin{align}
  \label{eq:rw_rewrite}
  \PP(T_2 < \tau_2) = 1 - \sum_{n=\pi(\tau_2) - 1}^\infty \PP\big(\min\{t \,:\, S_t = 1\} = n\big)
  ,
\end{align}
and similarly for $\PP(T_3 < \tau_3)$.

For any $z \in \Z_{>0}$,
the distribution of the smallest value of $t$ for which $S_t = z$ is given by
\begin{align*}
  \PP\big(\min\{t \,:\, S_t = z\} = n\big) = \frac{z}{n2^n}{n \choose \frac{z+n}{2}} \quad\text{if $z+n$ is even}
\end{align*}
and $0$ if $z+n$ is odd.
%
Using the effective form of Stirling's formula 
\begin{align*}
  \sqrt{2\pi n}\mspace{2mu}\frac{n^n}{e^n}e^{\frac{1}{12n}} e^{-\frac{1}{360n^3}} < n! < \sqrt{2\pi n}\mspace{2mu}\frac{n^n}{e^n}e^{\frac{1}{12n}}
  ,
\end{align*}
valid for all $n \geqslant 2$,
one finds that, for $N \geqslant 2$,
\begin{align}
  \label{eq:rw_approx}
  \PP\big(\min\{t \,:\, S_t = 1\} \geqslant 2N + 1\big)
  &= (1 + r) \sum_{n=N}^\infty \frac{1}{\sqrt{2\pi n(n+1)(2n+1)}}
\end{align}
with
\begin{align}
  \label{eq:rw_bound1}
  e^{-\frac{1}{360(2N+1)^3} - \frac{5}{24}\frac{2N+1}{N(N+1)}}
  < 1 + r <
  e^{\frac{1}{360}\left(\frac{1}{N^3} + \frac{1}{(N+1)^3}\right) + \frac{1}{24}\left(\frac{1}{N^2} - \frac{1}{(N+1)^2} - \frac{2N+1}{N(N+1)}\right)}
  .
\end{align}
The series in \eqref{eq:rw_approx} can be bounded as
\begin{align}
    \label{eq:rw_bound2}
  \frac{1}{\sqrt{16\pi(N+1)}}
  < \sum_{n=N}^\infty \frac{1}{\sqrt{2\pi n(n+1)(2n+1)}} <
  \frac{1}{\sqrt{16\pi(N-1)}}
  .
\end{align}



As Rubinstein and Sarnak \cite{rubinstein_sarnak} report,
\begin{align*}
  &\tau_2 = \min \{x \,:\, \pi(x;\,1 \mod 4) > \pi(x;\,3 \mod 4)\} = 26,\mspace{-2.5mu}861\\
  &\tau_3 = \min \{x \,:\, \pi(x;\,1 \mod 3) > \pi(x;\,2 \mod 3)\} = 608,\mspace{-2.5mu}981,\mspace{-2.5mu}813,\mspace{-2.5mu}029
  ,
\end{align*}
corresponding to $N = 1,\mspace{-2.5mu}472$ and $N = 11,\mspace{-2.5mu}669,\mspace{-2.5mu}295,\mspace{-2.5mu}395$ respectively. 

Combining \eqref{eq:rw_rewrite} with the bounds \eqref{eq:rw_bound1} and \eqref{eq:rw_bound2},
\begin{fact}
  \label{fact:chebyshev}
\begin{align}
  \nonumber
  \begin{aligned}
    (100 - 0.3674)\%  
    \,>\,{} &\PP(T_2 < \tau_2) \,>\,
    (100 - 0.3678)\% 
    \\
    (100 - 0.00013056980498)\% 
    \,>\,{} &\PP(T_3 < \tau_3) \,>\,
    (100 - 0.00013056980500)\% 
  .
  \end{aligned}
\end{align}
\end{fact}

\Cref{fact:chebyshev}
quantifies the qualitative observation that the first time at which primes $3 \mod 4$ outnumber primes $1 \mod 4$ occurs later than would be expected of unbiased random walks (it occurs later than about $99.6\%$ of random walks), and much more so for primes $2 \mod 3$ outnumbering primes $1 \mod 3$ ($99.9999\%$).
By quantifying the extent to which the values of \cref{statistic:chebyshev} are aberrant,
we obtain a concrete assessment of the extent to which primes mod $\q$ behave similarly or dissimilarly to the random model \eqref{eq:random_seq}.

\section{Relative entropy}

Our statistic of choice for discovering this paper's main result is the ``relative entropy'', also known as the Kullback--Leibler divergence or KL divergence, which we now define and motivate.

Let $\X_\alpha$ and $\X_\beta$ be discrete random variables with probability mass functions (PMFs) $\fp_\alpha$ and $\fp_\beta$ respectively. Define the random variable
\begin{align}
  \nonumber
  \log \PP(\hat{\X}_\alpha \mid \beta) \coloneqq \log \PP(\X_\beta = \X_\alpha) = \log \fp_\beta(\X_\alpha),
\end{align}
considered as a function of the random variable $\X_\alpha$.
If one imagines that $\hat{\X}_\alpha$ is a realization of the random variable $\X_\alpha$, then $\log \PP(\hat{\X}_\alpha \mid \beta)$ is the $\log$ of the probability that $\hat{\X}_\alpha$ would have been generated by $\X_\beta$.
The expectation of $\log \PP(\hat{\X}_\alpha \mid \beta)$ is 
\begin{align}
  \label{eq:loglikelihood_expectation}
  \E\big[\log \PP(\hat{\X}_\alpha \mid \beta)\big] &= \sum_a \fp_{\alpha}(a) \log \fp_{\beta}(a)
  \\
  \nonumber
  &\eqqcolon -H(\fp_{\alpha}, \fp_{\beta}).
\end{align}
The cross-entropy $H(\fp,\fq)$ is related to the 
\textit{relative entropy} $\dkl(\fp\parallel\fq)$ 
via \cite[\S 2.3]{cover_thomas}
\begin{align}
  \label{eq:KL_divergence_def}
  \dkl(\fp \parallel \fq)
  \coloneqq H(\fp,\fq) - H(\fp)
  =
  \sum_a \fp_{\alpha}(a) \log \frac{\fp_{\alpha}(a)}{\fp_\beta(a)}
  ,
\end{align}
where 
$H(\fp) \coloneqq H(\fp,\fp)$ denotes entropy.

It follows from e.g.\ Jensen's inequality that $\dkl(\fp \parallel \fq) \geqslant 0$ with equality iff $\fp = \fq$. Hence, the normalization \eqref{eq:KL_divergence_def} is the affine transformation of the expected log-likelihood \eqref{eq:loglikelihood_expectation} for which the transformed quantity is non-negative and actually $0$ sometimes.

The relevant case for us will be the one in which $\X_\beta = \unif$
and $\X_\alpha = M_x$ the empirical distribution of the tuple
\begin{align*}
  \left(
  \frac{\#\{p \leqslant x \,:\, \xi_{p,\q} = a_1\}}{\pi(x)}
  ,
  \frac{\#\{p \leqslant x \,:\, \xi_{p,\q} = a_2\}}{\pi(x)}
  ,
  \dots
  ,
  \frac{\#\{p \leqslant x \,:\, \xi_{p,\q} = a_{\varphi(\q)}\}}{\pi(x)}
  \right)\!
  ,
\end{align*}
with $a_1, a_2,\dots,a_{\varphi(\q)}$ ranging over $a \in \zqz$.
As $\xi_{p,\q} \sim \unif$ 
each independent, $M_x$ is multinomially distributed with $\pi(x)$ trials and all probabilities equal to $1/\varphi(\q)$ --- $\varphi(\q) = \#\zqz$ is Euler's totient function. For $a \in \zqz$, let $M_x(a)$ denote the $a^{\text{th}}$ component of this multinomial, i.e.
\begin{align*}
  M_x(a) = \frac{1}{\pi(x)}\sum_{p \leqslant x} \begin{cases} 1 & \text{if $\xi_{p,\q} = a$} \\ 0 & \text{otherwise.} \end{cases}
\end{align*}

\begin{statistic}
  \label{statistic:KL}
  On the arithmetic side, define
  \begin{align*}
    \tau_\q \coloneqq \dkl\!\left( \pi(x;\,\cdot\mod\q) \parallel \unif \right)
    &= \sum_{a \in \zqz} \frac{\pi(x; a\mod\q)}{\pi(x)} \log
    \!\left(
    \frac{\pi(x; a\mod\q)}{\pi(x)/\varphi(\q)}
    \right)
    ,
  \shortintertext{
    to be compared on the statistical side with
  }
    T_\q \coloneqq \dkl\!\left(M_x \parallel \unif\right)
    &= 
    \sum_{a \in \zqz} M_x(a) \log\!\big(\varphi(\q)M_x(a)\big)
    .
  \end{align*}
\end{statistic}
Our task is now to estimate the percentile $\PP(T_\q < \tau_\q)$. We need information about 
\begin{enumerate}[label=(\roman*)]
\item
  The distributions of the random statistics $T_\q$, and
\item
  The values of the empirical statistics $\tau_\q$.
\end{enumerate}

\section{The statistical aspect}

To set our expectations, the following remark classifies all possible limiting behaviours of relative entropies of multinomials.

\begin{remark}
  \label{rem:wilks}
  
  Let $\fp$ and $\fq$ be PMFs on some common finite set of size $k$,
  and let $M_n$ be multinomially distributed with $n$ trials and probabilities $\fp$.
  The strong law of large numbers implies that, as $n \to \infty$,
  \begin{align}
    \label{eq:lln_KL}
    \dkl(M_n \parallel \fq) \stackrel{\mathrm{a.s.}}{\longrightarrow} \dkl(\fp \parallel \fq)
    .
  \end{align}

  When $\fp = \fq$, Wilks' theorem \cite{wilks} states that, as $n \to \infty$,
  \begin{align}
    \label{eq:wilks}
    2n \dkl(M_n \parallel \fq) \longrightarrow \chi^2_{k-1},
  \end{align}
  a $\chi^2$-distributed random variable with $k-1$ degrees of freedom. 
  Above
  $\longrightarrow$ denotes convergence in distribution.

  Summarizing, the sum of relative entropies of $n$ iid random variables with PMF $\fp$, taken against $\fq$, is 
  $\gg n$ with probability $1$ if $\fp \neq \fq$, and otherwise, when rescaled by a factor of $2n$, converges in distribution to a $\chi^2$ random variable.
\end{remark}

To test our hypothesis that the arithmetic data \eqref{eq:arithmetic_seq} bears statistical semblance to the random model \eqref{eq:random_seq}, we evaluate for many moduli $\q$ the percentiles $\PP(T_\q < \tau_\q)$, where $T_\q$ and $\tau_\q$ are the relative entropies given in \cref{statistic:KL}. \Cref{rem:wilks} tells us that $2\pi(x)\cdot T_\q$ converges in distribution to a $\chi^2$ random variable with $\varphi(q)-1$ degrees of freedom,
and that no other iid choice of the random variables \eqref{eq:xidef} will do better.

The forthcoming
\cref{table:KL}
uncovers behaviour which qualitatively appears outlandish under our hypothesis. Comparing to the limiting distribution \eqref{eq:wilks} is known as a $G$-test, and is similar to Pearson's well-known $\chi^2$ goodness-of-fit test. But in order to establish rigorous bounds on the percentiles $\PP(T_\q < \tau_\q)$ an effective form of Wilks' theorem \eqref{eq:wilks} is needed,
specifically for the left tail.
This is quite an unusual situation to find oneself in in light of \cref{rem:wilks}: iid random variables can under no circumstances produce relative entropies much smaller than expected. Existing effective bounds in the literature, e.g.\ \cite{agrawal2020, agrawal2022, mjtnw, jvhw}, focus on the right tail, and aren't helpful here. 
To rectify the situation, we prove
the following technical theorem, which is practical for computation and strong enough to draw conclusions.

\begin{theorem}
  \label{thm:left_tail}
  Let
  \begin{itemize}
  \item
    $[y]$ denote the nearest integer to $y$ (with ties broken arbitrarily)
  \item
    $S$ be a finite set of size $k \geqslant 2$
  \item
    $M_n$ be a multinomially distributed random variable with $n \geqslant 2$ trials and all probabilities equal to $1/k$
  \item
    $\mu \coloneqq n/k$, 
  \item
    $\theta \in \R_{>0}$
  \item
    $B \coloneqq \lfloor n\sqrt{2\theta} + k|\mu - [\mu]| \rfloor$
  \item
    for any $\Delta \in \Z$, $0 < \Delta < [\mu]$,
    $$c_\Delta \coloneqq \frac{1}{2} + \frac{[\mu]}{\Delta} + \frac{[\mu]^2}{\Delta^2} \log\!\left(1 - \frac{\Delta}{[\mu]} \right)\!.$$ 
  \end{itemize}
  For all $\theta$ such that
  $0 < B < [\mu]$
  and
  ${\ds 2c_{B} > -\frac{2[\mu] - 1}{2[\mu] + 1}}$,
  \begin{align*}
    \PP\Big(\dkl
    &\big(M_n \parallel \mathrm{Unif}(S)\big) \leqslant \theta\Big)
    \\
    \,<\,{}&
    \frac{\sqrt{2\pi n} \,\mu^n}{(2\pi[\mu])^{\frac{k}{2}} [\mu]^n} e^{\frac{1}{12n} + k([\mu] - \mu)(1 + \frac{1}{2[\mu]})}
    \sum_{\Delta = 1}^B
    \exp\!\Bigg[
    \frac{\half - c_\Delta}{[\mu]^2}\Delta^3
    - \left(\frac{\thalf + c_\Delta}{[\mu]} - \frac{\thalf - c_\Delta}{2[\mu]^2}\right)\!\Delta
    \Bigg]
    \sum_{r=1}^{k} \binom{k}{r} \binom{\Delta - 1}{r - 1} 2^r
    \\
    &+ \frac{1}{k^n}
    \begin{cases}
      {\ds \frac{n!}{(\mu!)^k}} & \text{if $\mu \in \Z$}
      \\
      0 & \text{\emph{otherwise.}}
    \end{cases}
  \end{align*}
\end{theorem}

The proof of \cref{thm:left_tail} is relegated to \cref{sec:proof}.

\section{The arithmetic aspect}

We're now ready to assess the random model \eqref{eq:random_seq} by measuring the relative entropy as given in \cref{statistic:KL}. The empirical statistics $\tau_\q$ are straightforward to compute when $x = 10^8$, which is easily large enough for \cref{thm:left_tail}'s bound to reveal that the percentiles $\PP(T_\q < \tau_\q)$ are aberrant.
\Cref{table:KL}, which also includes scipy's estimate of $\PP(T_\q < \tau_\q)$, summarizes the situation.

\begin{table}[H]
  \begin{centering}
    \begin{tabular}{c|c|l|l c c|c|l|l}
      \multicolumn{2}{c}{} & \multicolumn{2}{c}{$\PP(T_\q < \tau_\q)$} & \multicolumn{3}{c}{} & \multicolumn{2}{c}{$\PP(T_\q < \tau_\q)$} \\
      $\q$ & \multicolumn{1}{c|}{$\tau_\q$} & \multicolumn{1}{c}{\Cref{thm:left_tail}}
      & \multicolumn{1}{c}{scipy}
      & \hspace{1cm}
      & $\q$ & $\tau_\q$ & \multicolumn{1}{c}{\Cref{thm:left_tail}}
      & \multicolumn{1}{c}{scipy \rule{0pt}{1em}}
      \\
      \cline{1-4} \cline{6-9}
  $3$  &  $2.66\cdot 10^{-9}$  &  $\leqslant 13.99\%$  &  $13.89\%$                         &  &   $22$  &  $7.08\cdot 10^{-8}$  &  $\leqslant 2.76\cdot 10^{-1}\%$  &  $2.42\cdot 10^{-2}\%$  \rule{0pt}{1em} \\
  $4$  &  $3.00\cdot 10^{-9}$  &  $\leqslant 14.86\%$  &  $14.74\%$                         &  &   $23$  &  $1.12\cdot 10^{-7}$  &  $\leqslant 2.06\cdot 10^{-5}\%$  &  $4.54\cdot 10^{-8}\%$  \\
  $5$  &  $3.95\cdot 10^{-9}$  &  $\leqslant 4.18\cdot 10^{-1}\%$  &  $2.55\cdot 10^{-1}\%$  &  &   $24$  &  $2.39\cdot 10^{-8}$  &  $\leqslant 3.90\cdot 10^{-2}\%$  &  $7.45\cdot 10^{-3}\%$  \\
  $6$  &  $2.64\cdot 10^{-9}$  &  $\leqslant 13.96\%$  &  $13.86\%$                         &  &   $25$  &  $1.00\cdot 10^{-7}$  &  $\leqslant 6.54\cdot 10^{-5}\%$  &  $2.78\cdot 10^{-7}\%$  \\
  $7$  &  $2.41\cdot 10^{-8}$  &  $\leqslant 5.87\cdot 10^{-1}\%$  &  $1.96\cdot 10^{-1}\%$  &  &   $26$  &  $5.17\cdot 10^{-8}$  &  $\leqslant 6.20\cdot 10^{-3}\%$  &  $3.47\cdot 10^{-4}\%$  \\
  $8$  &  $1.32\cdot 10^{-8}$  &  $\leqslant 2.52\%$  &  $1.50\%$                           &  &   $27$  &  $9.58\cdot 10^{-8}$  &  $\leqslant 4.18\cdot 10^{-4}\%$  &  $3.28\cdot 10^{-6}\%$  \\
  $9$  &  $1.02\cdot 10^{-8}$  &  $\leqslant 6.87\cdot 10^{-2}\%$  &  $2.41\cdot 10^{-2}\%$  &  &   $28$  &  $7.28\cdot 10^{-8}$  &  $\leqslant 4.10\cdot 10^{-2}\%$  &  $2.05\cdot 10^{-3}\%$  \\
 $10$  &  $3.93\cdot 10^{-9}$  &  $\leqslant 4.11\cdot 10^{-1}\%$  &  $2.53\cdot 10^{-1}\%$  &  &   $29$  &  $1.90\cdot 10^{-7}$  &  $\leqslant 2.79\cdot 10^{-5}\%$  &  $5.48\cdot 10^{-9}\%$  \\
 $11$  &  $7.09\cdot 10^{-8}$  &  $\leqslant 2.77\cdot 10^{-1}\%$  &  $2.43\cdot 10^{-2}\%$  &  &   $30$  &  $1.53\cdot 10^{-8}$  &  $\leqslant 8.36\cdot 10^{-3}\%$  &  $1.64\cdot 10^{-3}\%$  \\
 $12$  &  $9.47\cdot 10^{-9}$  &  $\leqslant 1.53\%$  &  $9.28\cdot 10^{-1}\%$              &  &   $31$  &  $1.90\cdot 10^{-7}$  &  $\leqslant 3.96\cdot 10^{-6}\%$  &  $3.86\cdot 10^{-10}\%$  \\
 $13$  &  $5.18\cdot 10^{-8}$  &  $\leqslant 6.23\cdot 10^{-3}\%$  &  $3.48\cdot 10^{-4}\%$  &  &   $32$  &  $8.54\cdot 10^{-8}$  &  $\leqslant 1.50\cdot 10^{-3}\%$  &  $2.26\cdot 10^{-5}\%$  \\
 $14$  &  $2.42\cdot 10^{-8}$  &  $\leqslant 5.90\cdot 10^{-1}\%$  &  $1.97\cdot 10^{-1}\%$  &  &   $33$  &  $9.37\cdot 10^{-8}$  &  $\leqslant 3.49\cdot 10^{-5}\%$  &  $1.55\cdot 10^{-7}\%$  \\
 $15$  &  $1.53\cdot 10^{-8}$  &  $\leqslant 8.31\cdot 10^{-3}\%$  &  $1.65\cdot 10^{-3}\%$  &  &   $34$  &  $7.71\cdot 10^{-8}$  &  $\leqslant 7.02\cdot 10^{-4}\%$  &  $1.10\cdot 10^{-5}\%$  \\
 $16$  &  $3.79\cdot 10^{-8}$  &  $\leqslant 1.94\cdot 10^{-1}\%$  &  $3.52\cdot 10^{-2}\%$  &  &   $35$  &  $1.03\cdot 10^{-7}$  &  $\leqslant 8.05\cdot 10^{-7}\%$  &  $1.01\cdot 10^{-9}\%$  \\
 $17$  &  $7.72\cdot 10^{-8}$  &  $\leqslant 6.98\cdot 10^{-4}\%$  &  $1.10\cdot 10^{-5}\%$  &  &   $36$  &  $3.57\cdot 10^{-8}$  &  $\leqslant 8.01\cdot 10^{-4}\%$  &  $4.84\cdot 10^{-5}\%$  \\
 $18$  &  $1.02\cdot 10^{-8}$  &  $\leqslant 6.83\cdot 10^{-2}\%$  &  $2.40\cdot 10^{-2}\%$  &  &   $37$  &  $2.60\cdot 10^{-7}$  &  $\leqslant 3.27\cdot 10^{-6}\%$  &  $1.92\cdot 10^{-11}\%$  \\
 $19$  &  $1.01\cdot 10^{-7}$  &  $\leqslant 6.52\cdot 10^{-4}\%$  &  $4.97\cdot 10^{-6}\%$  &  &   $38$  &  $1.01\cdot 10^{-7}$  &  $\leqslant 6.48\cdot 10^{-4}\%$  &  $4.96\cdot 10^{-6}\%$  \\
 $20$  &  $1.13\cdot 10^{-8}$  &  $\leqslant 2.84\cdot 10^{-3}\%$  &  $5.68\cdot 10^{-4}\%$  &  &   $39$  &  $1.16\cdot 10^{-7}$  &  $\leqslant 3.52\cdot 10^{-6}\%$  &  $4.01\cdot 10^{-9}\%$  \\
 $21$  &  $4.25\cdot 10^{-8}$  &  $\leqslant 2.09\cdot 10^{-3}\%$  &  $1.23\cdot 10^{-4}\%$  &  &   $40$  &  $5.01\cdot 10^{-8}$  &  $\leqslant 2.71\cdot 10^{-5}\%$  &  $4.94\cdot 10^{-7}\%$  
    \end{tabular}
    \begin{tablecap}
      \label{table:KL}
      For $x = 10^8$ and $\tau_\q, T_\q$ as in \cref{statistic:KL}. Code is available at \cite{reallysmallgit}.
    \end{tablecap}
  \end{centering}
\end{table}


It's clear from \cref{table:KL} that the sequence \eqref{eq:arithmetic_seq} of primes mod $\q$ does not behave like a typical realization of the random model \eqref{eq:random_seq}. Specifically, the relative entropies are much smaller than they would be for iid uniform random variables. This phenomenon is distinct from Chebyshev's bias discussed in \cref{sec:chebyshev}. In light of \cref{rem:wilks}, biases away from uniform in the distribution of primes mod $\q$ lead to larger relative entropies. The fact that, despite Chebyshev's bias, the relative entropies are still too small can be understood as saying that whatever is causing this behaviour is an effect which is more impactful than Chebyshev's bias, at least in the range of $x,\q$ tabulated here.

\Cref{rem:wilks} also explains that atypically small relative entropies are impossible for iid random variables, i.e.\ any model of the primes which elucidates the results of \cref{table:KL} must capture that their residues mod $\q$ are not independent of one another.

\section{The theoretical state of affairs}

Theoretical results can lend credence to the notion that \eqref{eq:arithmetic_seq} and \eqref{eq:random_seq} ought to be statistically similar.
For instance,
the \textit{strong law of large numbers} \cite[Thm.\ 2.4.1]{durrett}
\begin{align}
  \label{eq:lln}
  &\frac{\#\{p \leqslant x \,:\, \xi_{p,\q} = a\}}{\pi(x)} \stackrel{\mathrm{a.s.}}{\longrightarrow} \frac{1}{\varphi(\q)}
\end{align}
is mirrored
in the arithmetic by 
\textit{Dirichlet's theorem for primes in arithmetic progressions} \cite[Cor.\ 11.19]{MV}
\begin{align}
  \label{eq:dirichlet_ap}
  &\frac{\pi(x;\,a\mod\q)}{\pi(x)} \sim \frac{1}{\varphi(\q)}
  .
\end{align}
Furthering the analogy, the exponents in the rates of convergence in \eqref{eq:lln} and \eqref{eq:dirichlet_ap} match:
the consequence of Berry--Esseen's effective \textit{central limit theorem} \cite[Thms.\ 3.4.1 and 3.4.17]{durrett}
\begin{align*}
  &\left|\frac{\#\{p \leqslant x \,:\, \xi_{p,\q} = a\}}{\pi(x)} - \frac{1}{\varphi(\q)}\right| \ll \pi(x)^{-\half + \eps},
  \\
  &\left|\frac{\#\{p \leqslant x \,:\, \xi_{p,\q} = a\}}{\pi(x)} - \frac{1}{\varphi(\q)}\right| \gg \pi(x)^{-\half - \eps} \;\text{ infinitely often}
\end{align*}
is, on the arithmetic side, the consequence of the \textit{Riemann Hypothesis for Dirichlet $L$-functions} \cite[Cor.\ 13.8, \S 15]{MV}
\begin{align*}
  &\left|\frac{\pi(x;\,a \mod\q)}{\pi(x)} - \frac{1}{\varphi(\q)}\right| \ll \pi(x)^{-\half + \eps},
  \\
  &\left|\frac{\pi(x;\,a \mod\q)}{\pi(x)} - \frac{1}{\varphi(\q)}\right| \gg \pi(x)^{-\half - \eps} \;\text{ infinitely often}
  .
\end{align*}

Models like \eqref{eq:random_seq} naturally lead one to conjectures: absent evidence to the contrary, one may guess that statements about the random model true
$100\%$ of the time should be true for primes mod $\q$ as well.
For example, Hooley \cite{hooley} conjectures that as $x \geqslant \q \to \infty$ jointly in any manner,
\begin{align}
  \label{eq:hooley}
  \mathrm{Var}\Big[\pi(x;\,\cdot\mod\q)\Big] \sim x\log\q
  .
\end{align}
Various results support
this conjecture in the regime $\q \gg (\log\log x)^{1+\eps}$ \cite{fiorilli:primesap},  
and recently the conjecture was shown to be false in the complementary regime $\q \ll \log\log x$ \cite{fiorilli_martin}.
It's likely no coincidence that this precisely mirrors the behaviour of the random model \eqref{eq:random_seq} \cite[\S 8.5]{durrett}, making this outcome predictable. 


The results of \cref{table:KL} are not similarly predictable: Hooley's conjecture \eqref{eq:hooley} predicts variance among the values of $\pi(x;\,a\mod\q)$
that is too small for the iid random variables \eqref{eq:random_seq},
and was proven to be false by establishing a lower bound on said variance. In contrast, the statistics tabulated above, rooted in empirical fact, show variance among those same values $\pi(x;\,a\mod\q)$ implausibly small under the random model \eqref{eq:random_seq}! It appears that, while \eqref{eq:hooley} is false, the general notion is correct: primes mod $\q$ vary less than iid random variables would.

Though much work has been done on lower bounds --- see \cite{dlb_f} for a survey --- upper bounds elucidating \cref{table:KL}'s non-random behaviour are lacking. As noted, these would necessarily demonstrate interdependence between primes mod $\q$.



\appendix

\section{Proof of \cref*{thm:left_tail}}
\label{sec:proof}

Our proof of \cref{thm:left_tail} proceeds via the ``method of types'' \cite[\S 11]{cover_thomas}. Given a finite set $S = \{a_1,\dots,a_k\}$ of cardinality $\#S = k$ and a finite sequence $x = (x_1, x_2, \dots, x_n) \in S^n$ of length $n$ whose entries are elements of $S$, 
let $\fp_x$ denote the empirical distribution on $S$ yielded by $x$, i.e.
\begin{align*}
  \fp_x \coloneqq \left(\frac{\#\{\ell \,:\, x_\ell = a_1\}}{n}, \dots, \frac{\#\{\ell \,:\, x_\ell = a_k\}}{n}\right)\!.
\end{align*}
In the opposite direction, given a probability distribution $\fp$ on $S$ whose entries are rational numbers with denominator $n$, the set $T = T_\fp$ of all sequences $x$ of length $n$ for which $\fp_x = \fp$ is called the \textit{type class} of $\fp$. Write $T(a) \coloneqq n\fp(a)$ for $a \in S$.

Certain statistics of sequences $x \in S^n$, including the relative entropy $\dkl(x\parallel \fq)$ for fixed $\fq$, depend only on the type class of $\fp_x$. Ergo the study of the distribution of these statistics for $x \sim \mathrm{Unif}(S^n)$ can be reduced to combinatorics involving type classes.

The first step in proving \cref{thm:left_tail} invokes Pinsker's inequality
\begin{align}
  \label{eq:pinsker}
  \half \left(\sum_{a \in S} |\fp(a) - \fq(a)|\right)^{\!2} \leqslant \dkl(\fp\parallel\fq)
  .
\end{align}
With \eqref{eq:pinsker}, the question is transformed into one about $L^1$ norms of type classes, which is more amenable to combinatorial analysis.

The necessary bound on cardinalities of type classes is given in \cref{lemma:cTDelta}, while the weights which appear, multinomial coefficients, are estimated using an effective form of Stirling's formula and a polynomial bound on $\log(1 + x)$ coming from a series expansion.

\begin{proof}[Proof of \cref{thm:left_tail}]

%
Pinsker's inequality \eqref{eq:pinsker} implies that, for $x \sim \mathrm{Unif}(S^n)$,
\begin{align}
  \label{eq:pinsker_consequence}
  \PP(\dkl(\fp_x \parallel \mathrm{Unif}(S)) \leqslant \theta)
  \,\leqslant\,
  \PP\!\left(\sum_{a \in S} \left|\fp_x(a) - \tfrac{1}{k}\mspace{-1mu}\right| \leqslant \sqrt{2\theta} \right)\!
  .
\end{align}

Moving to type classes, let
\begin{align*}
  \cT
  &\coloneqq \left\{\text{type classes $T$} \,:\, T = T_{\fp_x} \text{ for some $x \in S^n$}\right\}
  \\
  &= \left\{ T = (T(a_1),\dots,T(a_k)) \,:\, T(a) \in \Z_{\geqslant 0},\, \sum_{a \in S} T(a) = n \right\}\!
  .
\end{align*}
The probability distribution on $\cT$ induced by $x \sim \mathrm{Unif}(S^n)$ is the uniform multinomial
\begin{align}
  \label{eq:type_pmf}
  \PP\!\left(T_{\fp_x} = T \mid x \sim \mathrm{Unif}(S^n)\right)
  =
  \frac{1}{k^n}
  \binom{n}{T(a_1), \dots, T(a_k)}
  =
  \frac{1}{k^n}
  \frac{n!}{T(a_1)!\cdots T(a_k)!}
  .
\end{align}
Let $\PP_\cT$ denote probabilities under the distribution \eqref{eq:type_pmf},
and set $\mu \coloneqq n/k$. Combining \eqref{eq:type_pmf} with \eqref{eq:pinsker_consequence},
\begin{align}
  \label{eq:dkl_bound1}
  \PP(\dkl(\fp_x \parallel \mathrm{Unif}(S)) \leqslant \theta)
  \,\leqslant\,
  \PP_\cT\!\left(\sum_{a \in S} |T(a) - \mu| \leqslant n\sqrt{2\theta} \right)\!
  .
\end{align}

The bound
\begin{align}
  \nonumber
  &\sum_{a \in S} |T(a) - \mu| \leqslant n\sqrt{2\theta}
  \\
  \nonumber
  \Longrightarrow\quad
  &\sum_{a \in S} |T(a) - [\mu]| \leqslant n\sqrt{2\theta} + k|\mu - [\mu]|
\end{align}
applied to the right hand side of \eqref{eq:dkl_bound1} gives
\begin{align}
  \label{eq:count_bound}
  \PP_\cT\!\left(\sum_{a \in S} |T(a) - \mu| \leqslant n\sqrt{2\theta} \right)
  \,\leqslant\,
  \PP_\cT\!\left(\sum_{a \in S} |T(a) - [\mu]| \leqslant n\sqrt{2\theta} + k|\mu - [\mu]| \right)\!
  .
\end{align}
Set $B = \lfloor n\sqrt{2\theta} + k|\mu - [\mu]| \rfloor$.
Via \eqref{eq:count_bound}, our objective is reduced to enumerating type classes with $L^1$ norm at most $B$, weighted by the right hand side of \eqref{eq:type_pmf}.

  Define
  \begin{align}
    \nonumber
    d_j &\coloneqq T(a_j) - [\mu]
    \\
    \nonumber
    \cT_\Delta &\coloneqq \left\{T \,:\, \sum_{j=1}^k |d_j| = \Delta \right\}
    \\
    \label{eq:cTB_def}
    \cT_{\leqslant B}
    &\coloneqq \bigcup_{\Delta = 0}^B \cT_\Delta
    = \left\{ T \,:\, \sum_{a \in S} |T(a) - [\mu]| \leqslant B \right\}\!
    .
  \end{align}

  Combining \eqref{eq:count_bound}, \eqref{eq:dkl_bound1}, and \eqref{eq:type_pmf}, the probability on the left hand side of \eqref{eq:dkl_bound1} which we're ultimately looking to bound, is in turn bounded via
  \begin{align}
    \label{eq:dkl_bound2}
    k^n \PP(\dkl(\fp_x \parallel \mathrm{Unif}(S)) \leqslant \theta)
    \,\leqslant\,
    \sum_{T \in \cT_{\leqslant B}} \frac{n!}{T(a_1)! \cdots T(a_k)!}
    .
  \end{align}
  The bound we obtain for \eqref{eq:dkl_bound2}'s right hand side, divided by $k^n$, will be \cref{thm:left_tail}'s right hand side.

  Let's separate the set of type classes $\cT_{\leqslant B}$ into the union \eqref{eq:cTB_def}, each piece to be bounded separately.
  \begin{align}
    \nonumber
    \sum_{T \in \cT_{\leqslant B}} \frac{n!}{T(a_1)! \cdots T(a_k)!}
    ={}&
    \sum_{T \in \cT_{\leqslant B}} \frac{n!}{\prod_{j=1}^k (d_j + [\mu])!}
    \\
    \nonumber
    ={}&
    \sum_{\Delta=0}^{B} \sum_{T \in \cT_\Delta} \frac{n!}{\prod_{j=1}^k (d_j + [\mu])!}
    \\
    ={}&
    \label{eq:dkl_bound3}
    \sum_{\Delta=1}^{B} \sum_{T \in \cT_\Delta} \frac{n!}{\prod_{j=1}^k (d_j + [\mu])!}
    + 
    \begin{cases}
      {\ds \frac{n!}{(\mu!)^k}} & \text{if $\mu \in \Z$}
      \\
      0 & \text{otherwise.}
    \end{cases}
  \end{align}
  Applying the effective version of Stirling's formula
  \begin{align*}
    \sqrt{2\pi n}\mspace{2mu}\frac{n^n}{e^n} < n! < \sqrt{2\pi n}\mspace{2mu}\frac{n^n}{e^n} e^{\frac{1}{12n}}
    ,
  \end{align*}
  valid for all $n \geqslant 1$,
  to each summand of \eqref{eq:dkl_bound3} yields
  \begin{align}
    \label{eq:dkl_bound4}
    \begin{aligned}
    \sum_{\Delta=1}^{B} \sum_{T \in \cT_\Delta} \frac{n!}{\prod_{j=1}^k (d_j + [\mu])!}
    &<
    \sum_{\Delta=1}^{B} \sum_{T \in \cT_\Delta}
    \exp\!\Bigg[
      n(\log n - 1) + \log \sqrt{2\pi} + \frac{1}{12n} + \frac{1}{2} \log n
      \\
      &\hspace{2cm}
      - \sum_{j=1}^k (d_j + [\mu])\left(\log(d_j + [\mu]) - 1\right) + \log \sqrt{2\pi} + \half\log(d_j + [\mu])\Bigg]
    .
    \end{aligned}
  \end{align}
  Rewriting the relation ${\ds \sum_{a \in S} T(a) = n}$ as
  \begin{align*}
     \exp\!\left(-n + \sum_{j=1}^k (d_j + [\mu])\right) = 1,
  \end{align*}
  the upper bound \eqref{eq:dkl_bound4} becomes
  \begin{align}
    \nonumber
    \sum_{\Delta=1}^{B} \sum_{T \in \cT_\Delta} \frac{n!}{\prod_{j=1}^k (d_j + [\mu])!}
    &<
    \sum_{\Delta=1}^{B} \sum_{T \in \cT_\Delta} (2\pi)^{\frac{k-1}{2}} e^{\frac{1}{12n}}
    \exp\!\Bigg[(n + \thalf) \log n - \sum_{j=1}^k (d_j + [\mu] + \thalf) \log(d_j + [\mu])\Bigg]
    \\
    \label{eq:dkl_bound5}
    &<
    \frac{\sqrt{2\pi n} \,n^n e^{\frac{1}{12n}}}{(2\pi)^{\frac{k}{2}} } \sum_{\Delta = 1}^B \sum_{T \in \cT_\Delta} 
    \exp\!\Bigg[ -\sum_{j=1}^k \left( d_j + [\mu] + \thalf \right) \log(d_j + [\mu]) \Bigg]
    .
  \end{align}


  Let's focus on the term $-\sum_{j=1}^k ([\mu] + \thalf)\log(d_j + [\mu])$ in the exponent of \eqref{eq:dkl_bound5}.
  For any $0 < R < 1$, define
  \begin{align*}
    c \coloneqq (R + \thalf R^2 + \log(1 - R))R^{-2}
    .
  \end{align*}
  For all $|x| < R$, 
  \begin{align*}
    \log(1 + x) \geqslant x - \big(\thalf - c\big) x^2
    .
  \end{align*}
  In the special case $R = \Delta/[\mu]$, which is less than $1$ by assumption, the value of $c$ is
  \begin{align*}
    c_\Delta \coloneqq \frac{1}{2} + \frac{[\mu]}{\Delta} + \frac{[\mu]^2}{\Delta^2} \log\!\left(1 - \frac{\Delta}{[\mu]} \right)\!
    .
  \end{align*}
  As $|d_j| \leqslant \Delta$ by definition of $\cT_\Delta$,
  \begin{align*}
    \sum_{j=1}^k \log(d_j + [\mu])
    &=
    k\log[\mu] + \sum_{j=1}^k \log\!\left(1 + \frac{d_j}{[\mu]}\right)
    \\
    &\geqslant
    k\log[\mu] + \sum_{j=1}^k \frac{d_j}{[\mu]} - \frac{\half - c_\Delta}{[\mu]^2}d_j^2
    \\
    &\geqslant
    k\log[\mu] + \frac{n - k[\mu]}{[\mu]} - \frac{\half - c_\Delta}{[\mu]^2} \sum_{j=1}^k d_j^2
    .
  \end{align*}
  Substituting into \eqref{eq:dkl_bound5},
  \begin{align}
    \nonumber
    \sum_{\Delta=1}^{B} \sum_{T \in \cT_\Delta} \frac{n!}{\prod_{j=1}^k (d_j + [\mu])!}
    &<
    \frac{\sqrt{2\pi n} \,n^n e^{\frac{1}{12n}}}{(2\pi)^{\frac{k}{2}} }
    \cdot
    \exp\!\left(-k([\mu] + \thalf)\left(\log[\mu] + \frac{\mu - [\mu]}{[\mu]}\right)\right)
    \\
    \nonumber
    &\quad\,\, \cdot
    \sum_{\Delta = 1}^B \sum_{T \in \cT_\Delta} 
    \exp\!\Bigg[
      - \sum_{j=1}^k d_j\log(d_j + [\mu])
      + (\thalf - c_\Delta)\left(\frac{1}{[\mu]} + \frac{1}{2[\mu]^2} \right) \sum_{j=1}^k d_j^2
      \Bigg]
    \\
    \label{eq:dkl_bound8}
    &\;
    \begin{aligned}
    <{}&
    \frac{\sqrt{2\pi n} \,n^n e^{\frac{1}{12n}}}{(2\pi[\mu])^{\frac{k}{2}} [\mu]^{k[\mu]}}
    \cdot
    \exp\!\left(k\left(1 + \frac{1}{2[\mu]}\right)([\mu] - \mu)\right)
    \\
    &\cdot
    \sum_{\Delta = 1}^B \sum_{T \in \cT_\Delta} 
    \exp\!\Bigg[
      - \sum_{j=1}^k d_j\log(d_j + [\mu])
      + (\thalf - c_\Delta)\left(\frac{1}{[\mu]} + \frac{1}{2[\mu]^2} \right) \sum_{j=1}^k d_j^2
      \Bigg]    
    .
    \end{aligned}
  \end{align}

  Looking at the exponent of \eqref{eq:dkl_bound8},
  \begin{align}
    \nonumber
    - \sum_{j=1}^k
    &d_j\log(d_j + [\mu])
    + (\thalf - c_\Delta)\left(\frac{1}{[\mu]} + \frac{1}{2[\mu]^2} \right) \sum_{j=1}^k d_j^2
    \\
    \nonumber
    &=
    - \sum_{j=1}^k d_j \left(\log[\mu] + \log\!\left(1 + \frac{d_j}{[\mu]}\right)\right)
    + (\thalf - c_\Delta)\left(\frac{1}{[\mu]} + \frac{1}{2[\mu]^2} \right) \sum_{j=1}^k d_j^2
    \\
    \nonumber
    &\leqslant
    - (n - k[\mu])\log[\mu]
    - \sum_{j=1}^k d_j \left(\frac{d_j}{[\mu]} - \frac{\half - c_\Delta}{[\mu]^2}d_j^2\right)
    + (\thalf - c_\Delta)\left(\frac{1}{[\mu]} + \frac{1}{2[\mu]^2} \right) \sum_{j=1}^k d_j^2
    \\
    \label{eq:dkl_bound6}
    &\leqslant
    - k(\mu - [\mu])\log[\mu]
    + \frac{\half - c_\Delta}{[\mu]^2}\sum_{j=1}^k d_j^3
    + \left(\frac{-\thalf - c_\Delta}{[\mu]} + \frac{\thalf - c_\Delta}{2[\mu]^2} \right) \sum_{j=1}^k d_j^2
    .
  \end{align}
  
  The assumption
  \begin{align*}
    c_{B} > -\half\frac{1 - \frac{1}{2[\mu]}}{1 + \frac{1}{2[\mu]}}
  \end{align*}
  implies
  $${\ds \frac{\thalf - c_\Delta}{2[\mu]^2} - \frac{\thalf + c_\Delta}{[\mu]} < 0.}$$
  In conjunction with the facts
  \begin{align*}
    \sum_{j=1}^k d_j^3  \leqslant \Delta^3 \quad\text{and}\quad  \sum_{j=1}^k d_j^2 \geqslant \Delta
    ,
  \end{align*}
  the bound
  \eqref{eq:dkl_bound6}
  can be loosened to
  \begin{align}
    \label{eq:dkl_bound7}
    \begin{aligned}
    - \sum_{j=1}^k
    &d_j\log(d_j + [\mu])
    + (\thalf - c_\Delta)\left(\frac{1}{[\mu]} + \frac{1}{2[\mu]^2} \right) \sum_{j=1}^k d_j^2
    \\
    &\leqslant
    - k(\mu - [\mu])\log[\mu]
    + \frac{\half - c_\Delta}{[\mu]^2}\Delta^3
    - \left(\frac{\thalf + c_\Delta}{[\mu]} - \frac{\thalf - c_\Delta}{2[\mu]^2} \right) \Delta
    .
    \end{aligned}
  \end{align}

  The following combinatorial lemma lets us handle the sum over $T \in \cT_\Delta$ featuring in \eqref{eq:dkl_bound8}.
  \begin{lemma}
    \label{lemma:cTDelta}
    For $\Delta \geqslant 1$,
  \begin{align*}
    \#\cT_\Delta \leqslant \sum_{r=1}^{k} \binom{k}{r} \binom{\Delta - 1}{r - 1} 2^r
    .
  \end{align*}
  \end{lemma}
  \begin{proof}
    Suppose there are exactly $r$ values of $j$ for which $d_j \neq 0$. As $\Delta \geqslant 1$, we have $1 \leqslant r \leqslant k$. Each case is mutually exclusive, so we count them separately and sum. The number of ways to distribute $m$ balls among $\ell$ bins is
      $\binom{m+\ell-1}{\ell - 1}\!\!$
      ,
    from which it can be deduced that the number of ways to write $\Delta$ as a sum of $r$ strictly positive integers is
      $\binom{\Delta - 1}{r - 1}\!\!$
      .
    Multiplying by $\binom{k}{r}$ to reflect the possible choices of precisely which $r$ values of $j$ have $d_j \neq 0$, multiplying by $2^r$ to reflect the possible sign choices, and summing over $1 \leqslant r \leqslant k$ yields \cref{lemma:cTDelta}.
  \end{proof}
  
  Applying \eqref{eq:dkl_bound7} and \cref{lemma:cTDelta} to \eqref{eq:dkl_bound8}, and then substituting into \eqref{eq:dkl_bound3} gives
  \begin{align*}
    \sum_{T \in \cT_{\leqslant B}}
    \frac{n!}{T(a_1)! \cdots T(a_k)!}
    <{}&
    \frac{\sqrt{2\pi n} \,n^n e^{\frac{1}{12n}}}{(2\pi[\mu])^{\frac{k}{2}} [\mu]^{n}}
    \cdot
    \exp\!\left(k\left(1 + \frac{1}{2[\mu]}\right)([\mu] - \mu)\right)
    \\
    &\cdot
    \sum_{\Delta = 0}^B
    \exp\!\Bigg[
    \frac{\half - c_\Delta}{[\mu]^2}\Delta^3
    - \left(\frac{\thalf + c_\Delta}{[\mu]} - \frac{\thalf - c_\Delta}{2[\mu]^2}\right)\Delta
    \Bigg]
    \sum_{r=1}^{k} \binom{k}{r} \binom{\Delta - 1}{r - 1} 2^r
    \\
    &+ 
    \begin{cases}
      {\ds \frac{n!}{(\mu!)^k}} & \text{if $\mu \in \Z$}
      \\
      0 & \text{otherwise.}
    \end{cases}
  \end{align*}
  Dividing by the total number of sequences $k^n$ to match \eqref{eq:dkl_bound2} yields \cref{thm:left_tail}.
\end{proof}

\section*{Acknowledgements}
We thank Noam Elkies, Kimball Martin, Michael Rubinstein, Kannan Soundararajan, Drew Sutherland, and Jerry Wang for helpful discussions.


\renewcommand{\bibliofont}{\normalfont\small} 
\bibliographystyle{amsalpha}
\bibliography{datapaperbib}{}

\end{document}